\documentclass[10pt]{amsart}
\usepackage{amsfonts}
\usepackage{ifthen}
\usepackage{amsthm}
\usepackage{amsmath,mathrsfs}
\usepackage{graphicx}
\usepackage{amscd,amssymb,amsthm}
\usepackage{color}
\usepackage{hyperref}
\usepackage{array,multirow,makecell}
\def\Li{\hbox{  Li}}
\setcellgapes{1pt}
\makegapedcells
\newcolumntype{R}[1]{>{\raggedleft\arraybackslash }b{#1}}
\newcolumntype{L}[1]{>{\raggedright\arraybackslash }b{#1}}
\newcolumntype{C}[1]{>{\centering\arraybackslash }b{#1}}
\newcounter{minutes}
\divide\time by 60
\newcounter{hours}
\multiply\time by 60 \addtocounter{minutes}{-\time}

\newtheorem{theorem}{Theorem}
\newtheorem{lemma}{Lemma}

\newtheorem{corollary}{Corollary}
\newtheorem{remark}{Remark}

\newtheorem{conj}{Conjecture}

\def\Li{\hbox{  Li}}
\title[Generating functions  involving the Fox-Wright functions]{On a new generating functions for the Fox-Wright functions and theirs applications}

\author[K. Mehrez]{Khaled Mehrez}
\address{Khaled Mehrez. D\'epartement de Math\'ematiques, Facult\'e de Sciences de Tunis, Universit\'e Tunis El Manar, Tunisia.}
\email{k.mehrez@yahoo.fr}



\begin{document}

\def\thefootnote{}
\footnotetext{ \texttt{File:~\jobname .tex,
          printed: \number\year-0\number\month-\number\day,
          \thehours.\ifnum\theminutes<10{0}\fi\theminutes}
} \makeatletter\def\thefootnote{\@arabic\c@footnote}\makeatother

\maketitle

\begin{abstract}
 The main focus of the present paper is to investigate several generating functions for a certain classes of functions associated to the Fox-Wright functions. In particular, certain generating functions for a class of function involving the Fox-Wright functions will be expressed in terms of the H-function of two variables are investigated. As applications, some generating functions associated to the generalized Mathieu type power series and  the extended Hurwitz-Lerch zeta function are established. Furthermore, some new double series identity are considered. A conjecture about the finite Laplace transform of a class of function associated to the Fox's H-function is made.
\end{abstract}

\section{ Introduction}
Throughout the present investigation, we use the following standard notations:
$$\mathbb{N}:=\left\{1,2,3,...\right\},\;\;\mathbb{N}_0:=\left\{0,1,2,3,...\right\}$$
and
$$\mathbb{Z}^-:=\left\{-1,-2,-3,...\right\}.$$
Also, as usual, $\mathbb{Z}$ denotes the set of integers, $\mathbb{R}$ denotes the set of real numbers, $\mathbb{R}^+$
denotes the set of positive numbers and $\mathbb{C}$ denotes the set of complex numbers.

Here, and in what follows, we use  ${}_p\Psi_q[.]$ to denote the Fox-Wright  generalization of the familiar hypergeometric  ${}_p F_q$  function with $p$ numerator and $q$ denominator parameters, defined by \cite[p. 4, Eq. (2.4)]{Z}
\begin{equation}\label{11}
\begin{split}
{}_p\Psi_q\Big[_{(b_1,B_1),...,(b_q,B_q)}^{(a_1,A_1),...,(a_p,A_p)}\Big|z \Big]&={}_p\Psi_q\Big[_{({\bf b}_q, {\bf B}_q)}^{({\bf a}_p, {\bf A}_p)}\Big|z \Big]\\
&=\sum_{k=0}^\infty\frac{\prod_{l=1}^p\Gamma(a_l+kA_l)}{\prod_{l=1}^q\Gamma(b_l+kB_l)}\frac{z^k}{k!},
\end{split}
\end{equation}
where,
$$(A_l\geq0,\; l=1,...,p; B_l\geq0,\;\textrm{and}\;l=1,...,q).$$
The convergence conditions and convergence radius of the series at the right-hand side of (\ref{11}) immediately follow from the known asymptotics of the Euler Gamma--function. The defining series in (\ref{11}) converges in the whole complex $z$-plane when
\begin{equation}
\Delta=\sum_{j=1}^q B_j-\sum_{i=1}^p A_i>-1.
\end{equation}
If $\Delta=-1,$ then the series in (\ref{11}) converges for $|z|<\rho,$ and $|z|=\rho$ under the condition $\Re(\mu)>\frac{1}{2},$ (see \cite{24} for details), where 
\begin{equation}
\rho=\left(\prod_{i=1}^p A_i^{-A_i}\right)\left(\prod_{j=1}^q B_j^{B_j}\right),\;\;\mu=\sum_{j=1}^q b_j-\sum_{k=1}^p a_k+\frac{p-q}{2}
\end{equation}

Throughout this paper, we denote by 
\begin{equation}\label{definition}
{}_{p+1}\check{\Psi}_q\Big[_{({\bf b}_q, {\bf B}_q)}^{(\sigma, 1), ({\bf a}_p, {\bf A}_p)}\Big|z \Big]=\left(\frac{1}{\Gamma(\sigma)}\right)\;{}_{p+1}\Psi_q\Big[_{({\bf b}_q, {\bf B}_q)}^{(\sigma, 1), ({\bf a}_p, {\bf A}_p)}\Big|z \Big].
\end{equation}
The generalized hypergeometric function ${}_p F_q$ is defined by 
\begin{equation}\label{12}
{}_p F_q\left[^{a_1,...,a_p}_{b_1,...,b_q}\Big|z\right]=\sum_{k=0}^\infty\frac{\prod_{l=1}^p(a_l)_k}{\prod_{l=1}^q(b_l)_k}\frac{z^k}{k!}
\end{equation}
where, as usual, we make use of the following notation:
$$(\tau)_0=1, \textrm{and}\;\;(\tau)_k=\tau(\tau+1)...(\tau+k-1)=\frac{\Gamma(\tau+k)}{\Gamma(\tau)},\;k\in\mathbb{N},$$
to denote the shifted factorial or the Pochhammer symbol. Obviously, we find from the definitions (\ref{11}) and (\ref{12}) that 
\begin{equation}\label{13}
{}_p\Psi_q\Big[_{(b_1,1),...,(b_q,1)}^{(a_1,1),...,(a_p,1)}\Big|z \Big]=\frac{\Gamma(a_1)...\Gamma(a_p)}{\Gamma(b_1)...\Gamma(b_q)}{}_p F_q\left[^{a_1,...,a_p}_{b_1,...,b_q}\Big|z\right].
\end{equation}

 Generating functions play an important role in the
investigation of various useful properties of the sequences
which they generate. They are used to find certain properties and formulas for numbers and polynomials in a wide variety of research subjects, indeed, in modern combinatorics. In this regard, in fact, a remarkable large number of generating functions involving a variety of special functions have been developed by many authors \cite{Agarwal1, Agarwal2, Agarwal3, chen, SR0, SR1, SR2, SR3}.

In a recent papers \cite{45},\cite{46},\cite{47}, Mehrez have studied certain  advanced properties of the Fox-Wright function including its
new integral representations, the Laplace and Stieltjes transforms, Luke
inequalities, Tur\'an type inequalities and completely monotonicity property are derived. In particular, it was
shown there that the following Fox-Wright functions are completely monotone:
$${}_p\Psi_q\Big[_{(\beta_q,A)}^{(\alpha_p,A)}\Big|-z \Big],\;s>0,$$
$${}_{p+1}\Psi_q\Big[_{(\beta_q,1)}^{(\lambda,1),(\alpha_p,A_p)}\Big|\frac{1}{z} \Big]\;s>0,$$
and has proved that the Fox's H-function $H_{q,p}^{p,0}[.]$ constitutes the representing measure for the Fox-Wright function ${}_p\Psi_q[.]$, if $\mu>0,$ i.e., \cite[Theorem 1]{45}
\begin{equation}\label{motivation}
{}_p\Psi_q\Big[_{(\beta_q, B_q)}^{(\alpha_p, A_p)}\Big|z\Big]=\int_0^\rho e^{zt}H_{q,p}^{p,0}\left(t\Big|^{(B_q,\beta_q)}_{(A_p,\alpha_p)}\right)\frac{dt}{t}.
\end{equation}
when $\mu>0.$ Here, and in what follows, we use $H_{q,p}^{p,0}[.]$ to denote the Fox's H-function, defined by
\begin{equation}
H_{q,p}^{p,0}\left(z\Big|^{(B_q,\beta_q)}_{(A_p,\alpha_p)}\right)=\frac{1}{2i\pi}\int_{\mathcal{L}}\frac{\prod_{j=1}^p\Gamma (A_j s+\alpha_j)}{\prod_{k=1}^q\Gamma (B_k s+\beta_k)}z^{-s}ds,
\end{equation}
where $A_j, B_k>0$ and $\alpha_j,\beta_k$ are real.  The contour $\mathcal{L}$ can be either the left loop $\mathcal{L}_-$ starting at $-\infty+i\alpha$ and ending at $-\infty+i\beta$  for some $\alpha<0<\beta$ such that all poles of the integrand lie inside the loop, or the right loop $\mathcal{L}_+$  starting $\infty+i\alpha$ at and ending $\infty+i\beta$ and leaving all poles on the left, or the vertical line $\mathcal{L}_{ic},\;\Re(z)=c,$ traversed upward and leaving all poles of the integrand on the left. Denote the rightmost pole of the integrand by $\gamma:$
$$\gamma=\min_{1\leq j\leq p}(a_j/A_j).$$ 

The definition of the H-function is still valid when the $A_i$'s and $B_j$'s  are positive rational numbers. Therefore, the H-function contains, as special cases,  all of the functions which are expressible in terms of the
G-function. More importantly, it contains  the Fox-Wright generalized hypergeometric function defined in (\ref{11}), the generalized Mittag-Leffler functions, etc. For example, the function ${}_p\Psi_q[.]$ is one of these special case of H-function. By the definition (\ref{11}) it is easily extended to the complex plane as follows \cite[Eq. 1.31]{AA},
\begin{equation}\label{label}
{}_p\Psi_q\Big[_{(\beta_q,B_q)}^{(\alpha_p,A_p)}\Big|z \Big]=H_{p,q+1}^{1,q}\left(-z\Big|_{(0,1),(B_q,1-\beta_q)}^{(A_p,1-\alpha_p)}\right).
\end{equation}
The representation (\ref{label}) holds true only for positive values of the parameters $A_i$ and $B_j$.

The special case for which the  H-function reduces to the Meijer G-function is when $A_1=...=A_p=B_1=...=B_q=A,\;A>0.$ In this case,
\begin{equation}\label{,,,}
H_{q,p}^{m,n}\left(z\Big|^{(B_q,\beta_q)}_{(A_p,\alpha_p)}\right)=\frac{1}{A}G_{p,q}^{m,n}\left(z^{1/A}\Big|^{B_q}_{\alpha_p}\right).
\end{equation}

In our present investigation, certain generating functions for some classes of function related to the Fox-Wright functions will be evaluated in terms of the H-function of two variables. In order to present the results, we need the
definition of the H-function of two complex variables introduced earlier by Mittal
and Gupta \cite{1972}. The analysis developed here is based on the work of Saxena
and Nishimoto \cite{1994}, Saigo and Saxena \cite{1998}. The H-function of two variables is defined in terms of multiple Mellin-Barnes type contour integral as
\begin{equation}\label{2var}
\begin{split}
H\left[^x_y\right]&=H^{0,n_1:m_2,n_2;m_3,n_3}_{p_1,q_1:p_2,q_2;p_3,q_3}\left[^x_y\Big|^{(\alpha_{p_1};{\bf A}_{p_1}, {\bf a}_{p_1}):({\bf c}_{p_2}, \gamma_{p_2});({\bf E}_{p_3},{\bf e}_{p_3})}_{(\beta_{q_1};{\bf B}_{q_1}, {\bf b}_{q_1}):({\bf d}_{q_2}, \delta_{q_2});({\bf F}_{q_3},{\bf f}_{q_3})}\right]\\
&=-\frac{1}{4\pi^2}\int_{L_1}\int_{L_2}\phi(s,t)\phi_1(s)\phi_2(t) x^s y^t ds dt,
\end{split}
\end{equation}
where $x$ and $y$ are not equal to zero. For convenience the parameters  $(\alpha_{p_1};{\bf A}_{p_1}, {\bf a}_{p_1})$ and $({\bf c}_{p_2}, \gamma_{p_2})$ will abbreviate the sequence of the parameters $(\alpha_{1};A_{1}, a_{1}),...,(\alpha_{p_1}; A_{p_1},  a_{p_1})$ and $(c_1,\gamma_1),...,(c_{p_2},\gamma_{p_2})$  respectively, and similar meanings hold for the other parameters $(\beta_{q_1};{\bf B}_{q_1}, {\bf b}_{q_1})$ and $({\bf d}_{q_2}, \delta_{q_2})$ , etc. Here
\begin{equation}
\phi(s,t)=\frac{\prod_{i=1}^{n_1}\Gamma(1-\alpha_i+a_i s+ A_it)}{\left[\prod_{i=n_1+1}^{p_1}\Gamma(\alpha_i-a_i s- A_i t)\right]\left[\prod_{j=1}^{q_1}\Gamma(1-\beta_j+b_j s+ B_j t)\right]}
\end{equation}
\begin{equation}
\phi_1(s)=\frac{\left[\prod_{j=1}^{m_2}\Gamma(d_j-\delta_j s)\right]\left[\prod_{i=1}^{n_2}\Gamma(1-c_i+\gamma_i s)\right]}{\left[\prod_{j=m_2+1}^{q_2}\Gamma(1-d_j+\delta_j s)\right]\left[\prod_{i=n_2+1}^{p_2}\Gamma(c_i-\gamma_i s)\right]}
\end{equation}
\begin{equation}
\phi_2(t)=\frac{\left[\prod_{j=1}^{m_3}\Gamma(f_j-F_j t)\right]\left[\prod_{i=1}^{n_3}\Gamma(1-e_i+E_i t)\right]}{\left[\prod_{j=m_3+1}^{q_3}\Gamma(1-f_j+Fj t)\right]\left[\prod_{i=n_3+1}^{p_3}\Gamma(e_i-E_i t)\right]},
\end{equation}
where  $\alpha_i, \beta_j, c i, d_j, e_i$ and $f_j$  be complex numbers and associated coefficients $a_i, A_i, b_j, B_j, \gamma_i, \delta_j, E_i$ and $F_j$  be real and positive for the standardization purposes, such that
\begin{equation}\label{2.61}
\rho_1=\sum_{i=1}^{p_1} a_i+\sum_{i=1}^{p_2}\gamma_i-\sum_{j=1}^{q_1} b_j-\sum_{j=1}^{q_2}\delta_j\leq0,
\end{equation}
\begin{equation}\label{2.62}
\rho_2=\sum_{i=1}^{p_1} A_i+\sum_{i=1}^{p_2}E_i-\sum_{j=1}^{q_1} B_j-\sum_{j=1}^{q_2}F_j\leq0,
\end{equation}
\begin{equation}\label{2.63}
\begin{split}
\Omega_1&=-\sum_{i=n_1+1}^{p_1} a_i-\sum_{j=1}^{q_1} b_j+\sum_{j=1}^{m_2}\delta_j-\sum_{j=m_2+1}^{p_2}\delta_j+\sum_{i=1}^{n_2}\gamma_i-\sum_{i=n_2+1}^{p_2}\gamma_i>0,
\end{split}
\end{equation}
\begin{equation}\label{2.64}
\begin{split}
\Omega_2&=-\sum_{i=n_1+1}^{p_1} A_i-\sum_{j=1}^{q_1} B_j+\sum_{j=1}^{m_2}F_j-\sum_{j=m_2+1}^{p_2}F_j+\sum_{i=1}^{n_2}E_i-\sum_{i=n_2+1}^{p_2}E_i>0.
\end{split}
\end{equation}
The contour integral (\ref{2var}) converges absolutely under the conditions (\ref{2.61})--(\ref{2.64}) and defines an analytic function of two complex variables $x$ and $y$ inside the sectors given by
$$|\arg(x)|<\frac{\pi}{2}\Omega_1\;\;\textrm{and}\;\;|\arg(y)|<\frac{\pi}{2}\Omega_2,$$
the points $x=0$ and $y=0$ being tacitly excluded,  for details the reader is referred to the book by Srivastava et al. \cite{sm1982}.

This paper is organized as follows. In Section 2, we establish generating functions for some classes of functions related to the Fox-Wright function. In particular, certain generating functions for a class of function involving the Fox-Wright functions will be expressed in terms of the H-function of two variables. In Section 3, as applications pf the main results in the Section 2, some generating functions associated to the generalized Mathieu type power series and  the extended Hurwitz-Lerch zeta function are presented. Furthermore, some new double series identity are derived. In addition,  we present an open a conjecture, which may be of interest for further research.
\section{Generating functions involving some classes of functions involving the Fox-Wright functions}

Our aim in this section is to derive some new generating functions for some class of functions related to the Fox-Wright functions.  We first recall that a generalized binomial coefficient $\binom{\lambda}{\mu}$ may be defined (for real or complex parameters $\lambda$ and $\mu$) by
$$\binom{\lambda}{\mu}=\frac{\Gamma(\lambda+1)}{\Gamma(\mu+1)\Gamma(\lambda-\mu+1)},$$
so that, in the special case when $\mu=n\in\mathbb{N}_0$ we have
$$\binom{\lambda}{n}=\frac{\lambda(\lambda-1)...(\lambda-n+1)}{n!},\;\;(\;n\in\mathbb{N}_0\;).$$
Secondly, recall that the finite Laplace Transform of a continuous ( or an almost piecewise continuous) function
$f(t)$  is denoted by
\begin{equation}
\mathcal{L}_T f(t)=\int_0^T e^{-\xi t} f(\xi) d\xi,\;t\in\mathbb{R}.
\end{equation}

Our first main result reads as follows.

\begin{theorem}\label{T1}Let $\lambda>0,$ then the following generating function
\begin{equation}\label{99}
\sum_{k=0}^\infty \binom{\lambda+k-1}{k} {}_p\Psi_q\Big[_{({\bf b}_{q}+k{\bf B}_{q}, {\bf B}_{q})}^{({\bf a}_{p}+k{\bf A}_{p}, {\bf A}_{p})}\Big|z \Big]t^k=\mathcal{L}_\rho\left(\xi^{-1}(1-t \xi)^{-\lambda} H_{q,p}^{p,0}\left(\xi\Big|^{({\bf B}_q,{\bf b}_q)}_{({\bf A}_p,{\bf a}_p)}\right)\right)(-z),
\end{equation}
holds true for all $|t|<1$ and $z\in\mathbb{R}.$ Furthermore, suppose that the following conditions
$$(H_1): \mu>0,\;\; \gamma\geq1,\;\;\sum_{k=1}^p A_k=\sum_{j=1}^q B_j,\;H_{q,p}^{p,0}[z]\geq0,$$
are satisfied, then following inequalities 
\begin{equation}\label{YOU}
\begin{split}
\left(\frac{1}{\Gamma(\sigma)}\right).\;{}_{p+1}\Psi_q\left[^{(\sigma,1),(\alpha_p,A_p)}_{(\beta_q,B_q)}\Big|t\right]&\leq\sum_{k=0}^\infty \binom{\lambda+k-1}{k}{}_p\Psi_q\Big[_{({\bf b}_{q}+k{\bf B}_{q}, {\bf B}_{q})}^{({\bf a}_{p}+k{\bf A}_{p}, {\bf A}_{p})}\Big|z \Big]t^k\\&\leq \left(\frac{e^{\rho z}}{\Gamma(\sigma)}\right).\;{}_{p+1}\Psi_q\left[^{(\sigma,1),(\alpha_p,A_p)}_{(\beta_q,B_q)}\Big|t\right]
\end{split}
\end{equation}
hold true for all $|t|<1$ and $z>0.$
\end{theorem}
\begin{proof}For convenience, let the left-hand side of (\ref{99}) be
denoted by $\mathcal{I}.$ Applying the integral expression (\ref{motivation}) to $\mathcal{I},$ and  we  employ the formula \cite[Property 1.5]{AA}
\begin{equation}\label{ooo}
H_{q,p}^{p,0}\left(\xi\Big|^{({\bf B}_q,{\bf b}_q+k{\bf B}_q)}_{({\bf A}_p,{\bf a}_p+k{\bf A}_q)}\right)=\xi^kH_{q,p}^{p,0}\left(\xi\Big|^{({\bf B}_q,{\bf b}_q)}_{({\bf A}_p,{\bf a}_p)}\right),\;k\in\mathbb{C},
\end{equation}
we thus get
\begin{equation}\label{yu}
\begin{split}
\mathcal{I}&=\sum_{k=0}^\infty \binom{\lambda+k-1}{k}\left[\int_0^\rho e^{z\xi}H_{q,p}^{p,0}\left(\xi\Big|^{({\bf B}_q,{\bf b}_q+k{\bf B}_q)}_{({\bf A}_p,{\bf a}_p+k{\bf A}_q)}\right)\frac{d\xi}{\xi}\right]t^k\\
&=\sum_{k=0}^\infty \binom{\lambda+k-1}{k}\left[\int_0^\rho \xi^{k-1}e^{z\xi}H_{q,p}^{p,0}\left(\xi\Big|^{({\bf B}_q,{\bf b}_q)}_{({\bf A}_p,{\bf a}_p)}\right)d\xi\right]t^k\\
&=\int_0^\rho e^{z\xi}H_{q,p}^{p,0}\left(\xi\Big|^{({\bf B}_q,{\bf b}_q)}_{({\bf A}_p,{\bf a}_p)}\right)\left[\sum_{k=0}^\infty \binom{\lambda+k-1}{k} (t\xi)^k\right]\frac{d\xi}{\xi}
\end{split}
\end{equation}
Further, upon using the generalized binomial expansion 
\begin{equation}\label{!!}
\sum_{k=0}^\infty \binom{\lambda+k-1}{k} t^k=(1-t)^{-\lambda},\;|t|<1,
\end{equation}
for evaluating the inner sum in (\ref{yu}), we obtain the desired formula (\ref{99}). Now, let us focus on  the inequalities (\ref{YOU}). Since $\xi\in(0,\rho)$  it follows that
\begin{equation}\label{***}
\begin{split}
\int_0^\rho H_{q,p}^{p,0}\left(\xi\Big|^{(B_q,\beta_q)}_{(A_p,\alpha_p)}\right)\frac{d\xi}{\xi (1-t \xi)^{\lambda}}&\leq \mathcal{L}_\rho\left(\xi^{-1}(1-t \xi)^{-\lambda} H_{q,p}^{p,0}\left(\xi\Big|^{(B_q,\beta_q)}_{(A_p,\alpha_p)}\right)\right)(-z)\\
&\leq e^{\rho z}\int_0^\rho H_{q,p}^{p,0}\left(\xi\Big|^{(B_q,\beta_q)}_{(A_p,\alpha_p)}\right)\frac{d\xi}{\xi (1-t \xi)^{\lambda}}.
\end{split}
\end{equation}
By means of the integral representation \cite[Theorem 4]{45}
\begin{equation}\label{ij}
{}_{p+1}\Psi_q\left[^{(\sigma,1),(\alpha_p,A_p)}_{(\beta_q,B_q)}\Big|z\right]=\Gamma(\sigma)\int_0^\rho H_{q,p}^{p,0}\left(\xi\Big|^{(B_q,\beta_q)}_{(A_p,\alpha_p)}\right) \frac{d\xi}{\xi(1-z\xi)^\sigma},
\end{equation}
where $\sigma>0$ and $z\in\mathbb{C}$ such that $|z|<1,$ and the conditions $(H_1)$ are satisfied. Therefore, the inequalities (\ref{***}) transforms into the form
\begin{equation}\label{****}
\begin{split}
\left(\frac{1}{\Gamma(\sigma)}\right).\;{}_{p+1}\Psi_q\left[^{(\sigma,1),(\alpha_p,A_p)}_{(\beta_q,B_q)}\Big|t\right]&\leq \mathcal{L}_\rho\left(\xi^{-1}(1-t \xi)^{-\lambda} H_{q,p}^{p,0}\left(\xi\Big|^{(B_q,\beta_q)}_{(A_p,\alpha_p)}\right)\right)(-z)\\
&\leq\left(\frac{e^{\rho z}}{\Gamma(\sigma)}\right).\;{}_{p+1}\Psi_q\left[^{(\sigma,1),(\alpha_p,A_p)}_{(\beta_q,B_q)}\Big|t\right].
\end{split}
\end{equation}
Hence, in view of (\ref{***}) and (\ref{****}) we deduce that the inequalities (\ref{YOU}) hold true. The proof of Theorem \ref{T1} is completes.
\end{proof}
\begin{remark}Mehrez\cite[Formula (4.70)]{45} have further shown the following Luke type inequalities for the Fox-Wright function ${}_{p+1}\Psi_q[.],$ that is 
\begin{equation}\label{abel}
\frac{\psi_{0,0}}{\left(1+\frac{\psi_{0,1}}{\psi_{0,0}}z\right)^\lambda}\leq{}_{p+1}\Psi_p\left[_{\;\;\;(\beta_q,B_q)}^{(\lambda,1),(\alpha_p,A_p)}\Big|-z\right]\leq\left[\psi_{0,0}-\frac{\psi_{0,1}}{\rho}\left(1-\frac{1}{(1+\rho z)^\lambda}\right)\right],\;z\in\mathbb{R},\;\lambda>0,
\end{equation}
which holds under the conditions $(H_1),$  where 
$$\psi_{0,0}=\frac{\prod_{i=1}^p\Gamma(a_i)}{\prod_{j=1}^q\Gamma(b_j)},\;\textrm{and}\;\;\psi_{0,1}=\frac{\prod_{i=1}^p\Gamma(a_i+A_i)}{\prod_{j=1}^q\Gamma(b_j+B_j)}.$$ 
In virtue of (\ref{YOU}) and (\ref{abel}) we get the following inequalities
\begin{equation}\label{6666}
\frac{\psi_{0,0}}{\left(1-\frac{\psi_{0,1}}{\psi_{0,0}}t\right)^\lambda}\leq\sum_{k=0}^\infty \binom{\lambda+k-1}{k} {}_p\Psi_q\Big[_{({\bf b}_{q}+k{\bf B}_{q}, {\bf B}_{q})}^{({\bf a}_{p}+k{\bf A}_{p}, {\bf A}_{p})}\Big|z \Big]t^k\leq e^{\rho z}\left[\psi_{0,0}-\frac{\psi_{0,1}}{\rho}\left(1-\frac{1}{(1-\rho t)^\lambda}\right)\right]
\end{equation}
\end{remark}
\begin{conj} Motivated by the previous Theorem, we ask the following question: Proved the finite Laplace transform of the function
$$\xi\mapsto \frac{1}{\xi(1-t \xi)^{\lambda}} H_{q,p}^{p,0}\left(\xi\Big|^{(B_q,\beta_q)}_{(A_p,\alpha_p)}\right),\;(0<t,\xi<1),$$
or, did you express the finite Laplace transform of the above function in terms of the Fox H-function.
\end{conj}

On setting $p=q=1, A_1=B_1=A>0$  in (\ref{99}), in view of the following formula
\begin{equation}
H_{q,p}^{p,0}\left(\xi\Big|^{(A,\beta)}_{(A,\alpha)}\right)=\frac{\xi^{\frac{\alpha}{A}}(1-\xi^{\frac{1}{A}})^{\beta-\alpha-1}}{A\Gamma(\beta-\alpha)},\;A>0,\; \beta>\alpha>0,
\end{equation}
 we get the following results as follows:

\begin{corollary} If $b>a\geq A,$ then the following relation holds true:
\begin{equation}
\sum_{k=0}^\infty\sum_{n=0}^\infty \frac{\Gamma(\lambda+k)\Gamma(\alpha+(k+n)A)}{\Gamma(\beta+(k+m)A)}\frac{z^n t^k}{n!k!}=\frac{\Gamma(\lambda)}{A\Gamma(\beta-\alpha)}\mathcal{L}_1\left(\frac{\xi^{\frac{\alpha}{A}-1}(1-\xi)^{\beta-\alpha-1}}{(1-t\xi)^\lambda}\right)(-z),\;\;|t|<1.
\end{equation}
\end{corollary}

On taking $A_i=B_i=A>0$  in (\ref{99}), from (\ref{,,,}) we compute the following results as follows:

\begin{corollary}Let $\lambda>0,$ then the following generating function
\begin{equation}\label{000099}
\sum_{k=0}^\infty \binom{\lambda+k-1}{k} {}_p\Psi_q\Big[_{({\bf b}_{q}+k A, A)}^{({\bf a}_{p}+k A, A)}\Big|z \Big]t^k=\frac{1}{A}\mathcal{L}_\rho\left(\xi^{-1}(1-t \xi)^{-\lambda} G_{q,p}^{p,0}\left(\xi^{\frac{1}{A}}\Big|^{b_q}_{ a_p})\right)\right)(-z),
\end{equation}
holds true for all $|t|<1$ and $z\in\mathbb{R}$ such that $\min(a_i/A)\geq1$ and $\mu>0.$
\end{corollary}

If we set $A_i=B_i=1$ in (\ref{000099}), we get the following results as follows:

\begin{corollary}Let $\lambda>0,$ then the following generating function
\begin{equation}\label{00009900}
\sum_{k=0}^\infty \binom{\lambda+k-1}{k} {}_p F_q\Big[_{b_1+k,..., b_{q}+k }^{a_1+k,..., a_{p}+k }\Big|z \Big]t^k=\mathcal{L}_\rho\left(\xi^{-1}(1-t \xi)^{-\lambda} G_{q,p}^{p,0}\left(\xi\Big|^{b_q}_{ a_p})\right)\right)(-z),
\end{equation}
holds true for all $|t|<1$ and $z\in\mathbb{R}$ such that $\min(a_i)\geq1$ and $\mu>0.$
\end{corollary}
 We will need the following Lemma which is considered the main tool to arrive at our result in the next Theorem and Theorem \ref{TT3}.

\begin{lemma}\label{l1}\cite[Exercice 2.3, p. 72]{AA} Let $\alpha, \beta, \gamma_1\in\mathbb{C},$ either $\alpha>0, |\arg y|<\frac{1}{2}\pi \alpha$ or $\alpha=0, \Re(\mu)>1.$  Further, let $\eta\geq0, b\neq a, \left|\frac{(b-a)c}{ac+d}\right|<1, \left|\frac{y(b-a)^{\sigma+\eta}}{(ac+d)^\nu}\right|<1, \left|\arg\left(\frac{cb+d}{ca+d}\right)\right|<\pi$ be such that $\Re(\alpha)+\sigma\min_{1\leq i\leq n}\left(\frac{\Re(a_i)}{A_i}\right)>0, \Re(\alpha)+\eta\min_{1\leq i\leq n}\left(\frac{\Re(a_i)}{A_i}\right)>0$ for $\alpha>0$ or $\alpha=0, \Delta\leq0$ and $\Re(\alpha)+\sigma\min_{1\leq i\leq n}\left(\frac{\Re(a_i)}{A_i}, \frac{\Re(\mu)-1/2}{\Delta}\right)>0, \Re(\alpha)+\eta\min_{1\leq i\leq n}\left(\frac{\Re(a_i)}{A_i}, \frac{\Re(\mu)-1/2}{\Delta}\right)>0$ for $\alpha>0$ or $\alpha=0, \Delta>0,$ then there holds the formula
\begin{equation}
\begin{split}
\int_a^b (x-a)^{\alpha-1}(b-x)^{\beta-1}(cx+d)^{\gamma_1}& H_{q,p}^{m,n}\left(y(x-a)^\sigma(b-x)^\eta(cx+d)^{-\nu}\Big|^{({\bf B}_q,{\bf b}_q)}_{({\bf A}_p,{\bf a}_p)}\right)dx\\
=(b-a)^{\alpha+\beta-1}(ac+d)^{\gamma_1}
&\times H_{2,1:q+1,p+1;0,1}^{0,2:n+1,m;1,0}\left(^{\frac{y(b-a)^{\sigma+\eta}}{(ac+d)^\nu}}_{\frac{c(b-a)}{ac+d}}\Big|^{^{(1-\alpha;1,\sigma),(1+\gamma_1;1,\nu):-}_{({\bf B}_q, {\bf b}_q);(\eta,1-\beta)}}_{(1-\alpha-\beta;1,\sigma+\eta),({\bf A}_p, {\bf a}_p),(\nu,1+\gamma_1);(1,0)}\right).
\end{split}
\end{equation}
\end{lemma}
\begin{theorem}\label{T2}Let $\lambda>0,$ and assume that the hypotheses $(H_1)$ are satisfied. If $\mu>1,$ then there holds the generating function
\begin{equation}\label{98}
\sum_{k=0}^\infty \binom{\lambda+k-1}{k} {}_{p+1}\check{\Psi}_q\Big[_{\;\;\;\;\;\;({\bf b}_{q}, {\bf B}_{q})}^{(\lambda+k,1), ({\bf a}_{p}, {\bf A}_{p})}\Big|\frac{1-t}{\rho}\Big]t^k=(1-t)^{-\lambda}H_{2,1:q+1,p+1;0,1}^{0,2:p+1,1;1,0}\left(^{\rho}_{\rho}\Big|^{^{\;\;\;\;\;(1;1,1),(1;1,0):-}_{\;\;\;\;({\bf B}_q,{\bf b}_q);(0, \lambda)}}_{(\lambda;1,1),({\bf A}_p,{\bf a}_p),(0,1);(1,0)}\right),
\end{equation}
for all $1-t<\rho<1.$
\end{theorem}
\begin{proof} Upon setting the left hand-side of the formula (\ref{98}) by $\mathcal{J}.$  Then, by substituting the integral representation (\ref{ij}) into $\mathcal{J},$ we find that
\begin{equation}
\mathcal{J}=\sum_{k=0}^\infty \binom{\lambda+k-1}{k} \left[\int_0^\rho H_{q,p}^{p,0}\left(\xi\Big|^{({\bf B}_q,{\bf b}_q)}_{({\bf A}_p,{\bf a}_p)}\right) \frac{d\xi}{\xi (1-\xi(1-t)/\rho)^{\lambda+k}}\right]t^k,
\end{equation}
which, upon changing the order of sum and integral and after a little simplification when we make used (\ref{!!}), yields
\begin{equation}\label{ii}
\begin{split}
\mathcal{J}&=\int_0^\rho H_{q,p}^{p,0}\left(\xi\Big|^{({\bf B}_q,{\bf b}_q)}_{({\bf A}_p,{\bf a}_p)}\right) \left[\sum_{k=0}^\infty \binom{\lambda+k-1}{k}\left(\frac{t}{1-\xi(1-t)/\rho}\right)^k\right]\frac{d\xi}{\xi(1-\xi(1-t)/\rho)^{\lambda}}\\
&=\int_0^\rho H_{q,p}^{p,0}\left(\xi\Big|^{({\bf B}_q,{\bf b}_q)}_{({\bf A}_p,{\bf a}_p)}\right)\frac{d\xi}{\xi(1-t-\xi(1-t)/\rho)^{\lambda}}\\
&=\left(\frac{\rho}{1-t}\right)^\lambda\int_0^\rho \xi^{-1}(\rho-\xi)^{-\lambda}H_{q,p}^{p,0}\left(\xi\Big|^{({\bf B}_q,{\bf b}_q)}_{({\bf A}_p,{\bf a}_p)}\right)d\xi.
\end{split}
\end{equation}
Now, let us put $\alpha=\eta=a=\nu=\gamma_1=0, \beta=1-\lambda, d=c=y=1$ and $b=\rho$ in Lemma \ref{l1}, then we obtain
\begin{equation}\label{iii}
\int_0^\rho \xi^{-1}(\rho-\xi)^{-\lambda}H_{q,p}^{p,0}\left(\xi\Big|^{({\bf B}_q,{\bf b}_q)}_{({\bf A}_p,{\bf a}_p)}\right)d\xi=\rho^{-\lambda}H_{2,1:q+1,p+1;0,1}^{0,2:p+1,1;1,0}\left(^{\rho}_{\rho}\Big|^{^{\;\;\;\;\;(1;1,1),(1;1,0):-}_{\;\;\;\;({\bf B}_q,{\bf b}_q):(0, \lambda)}}_{(\lambda;1,1),({\bf A}_p,{\bf a}_p),(0,1):(1,0)}\right).
\end{equation}
Finally, in view of (\ref{ii}) and (\ref{iii}), we get the desired assertion (\ref{98}) of Theorem \ref{T2}.
\end{proof}

\begin{remark}By using (\ref{definition}) and under the hypotheses of Theorem \ref{T2} the formula (\ref{98}) can be rewritten as follows:
\begin{equation}\label{098}
\begin{split}
\sum_{k=0}^\infty  {}_{p+1}\Psi_q\Big[_{\;\;\;\;\;\;({\bf b}_{q}, {\bf B}_{q})}^{(\lambda+k,1), ({\bf a}_{p}, {\bf A}_{p})}\Big|\frac{1-t}{\rho}\Big]\frac{t^k}{k!}&=\frac{\Gamma(\lambda)}{(1-t)^{\lambda}}\\
&\times H_{2,1:q+1,p+1;0,1}^{0,2:p+1,1;1,0}\left(^{\rho}_{\rho}\Big|^{^{\;\;\;\;\;(1;1,1),(1;1,0):-}_{\;\;\;\;({\bf B}_q,{\bf b}_q);(0, \lambda)}}_{(\lambda;1,1),({\bf A}_p,{\bf a}_p),(0,1);(1,0)}\right).
\end{split}
\end{equation}
\end{remark}
\begin{theorem}\label{TT3}Let $\lambda>0, \tau>0$ and assume that the hypotheses $(H_1)$ are satisfied. Moreover, suppose that the following hypotheses
$$(H_2):\;\tau+\min_{1\leq j\leq p}(a_j/A_j, (\mu-1/2)/\Delta)>0,$$
are verified. Then the following generating function
\begin{equation}\label{CV}
\begin{split}
\sum_{k=0}^\infty \binom{\lambda+k-1}{k} {}_{p+1}\check{\Psi}_q\Big[_{\;\;\;\;\;({\bf b}_{q}+\tau{\bf B}_{q}, {\bf B}_{q})}^{(\lambda+k,1), ({\bf a}_{p}+\tau{\bf A}_{p}, {\bf A}_{p})}\Big|\frac{1-t}{\rho}\Big]t^k&=\frac{\rho^\tau}{(1-t)^{\lambda}}\\
&\times H_{2,1:q+1,p+1;0,1}^{0,2:p+1,1;1,0}\left(^{\rho}_{\rho}\Big|^{^{\;\;\;\;\;(1-\tau;1,1),(1;1,0):-}_{\;\;\;\;({\bf B}_q,{\bf b}_q);(0, \lambda)}}_{(\lambda-\tau;1,1),({\bf A}_p,{\bf a}_p),(0,1);(1,0)}\right),
\end{split}
\end{equation}
holds true for all $1-t<\rho<1.$
\end{theorem}
\begin{proof}Making use of (\ref{ooo}), (\ref{!!})  and (\ref{ij}), and then changing the order of integration and summation, the left-hand side of the result (\ref{CV}) (say $\mathcal{K},$)  it follows that
\begin{equation}\label{k!k}
\mathcal{K}=\left(\frac{\rho}{1-t}\right)^\lambda\int_0^\rho \xi^{\tau-1}(\rho-\xi)^{-\lambda}H_{q,p}^{p,0}\left(\xi\Big|^{({\bf B}_q,{\bf b}_q)}_{({\bf A}_p,{\bf a}_p)}\right)d\xi.
\end{equation}
Now, let us put $\eta=a=\gamma_1=\nu=0, \alpha=\tau, \beta=1-\lambda, d=c=y=1$ and $b=\rho$ in Lemma \ref{l1}, then we have
\begin{equation}\label{k!}
\int_0^\rho \xi^{\tau-1}(\rho-\xi)^{-\lambda}H_{q,p}^{p,0}\left(\xi\Big|^{({\bf B}_q,{\bf b}_q)}_{({\bf A}_p,{\bf a}_p)}\right)d\xi=\rho^{\tau-\lambda}H_{2,1:q+1,p+1;0,1}^{0,2:p+1,1;1,0}\left(^{\rho}_{\rho}\Big|^{^{\;\;\;\;\;(1-\tau;1,1),(1;1,0):-}_{\;\;\;\;({\bf B}_q,{\bf b}_q);(0, \lambda)}}_{(\lambda-\tau;1,1),({\bf A}_p,{\bf a}_p),(0,1);(1,0)}\right).
\end{equation}
Now on taking (\ref{k!k}) and (\ref{k!}) into account, one can easily arrive at
the desired result (\ref{CV}). This completes the proof.
\end{proof}

\begin{remark}By using (\ref{definition}) and under the hypotheses of Theorem \ref{TT3} the formula (\ref{CV}) can be rewritten as follows:
\begin{equation}\label{CVBIEN}
\begin{split}
\sum_{k=0}^\infty {}_{p+1}\Psi_q\Big[_{\;\;\;\;\;({\bf b}_{q}+\tau{\bf B}_{q}, {\bf B}_{q})}^{(\lambda+k,1), ({\bf a}_{p}+\tau{\bf A}_{p}, {\bf A}_{p})}\Big|\frac{1-t}{\rho}\Big]\frac{t^k}{k!}&=\frac{\Gamma(\lambda)\rho^\tau}{(1-t)^{\lambda}}\\
&\times H_{2,1:q+1,p+1;0,1}^{0,2:p+1,1;1,0}\left(^{\rho}_{\rho}\Big|^{^{\;\;\;\;\;(1-\tau;1,1),(1;1,0):-}_{\;\;\;\;({\bf B}_q,{\bf b}_q);(0, \lambda)}}_{(\lambda-\tau;1,1),({\bf A}_p,{\bf a}_p),(0,1);(1,0)}\right).
\end{split}
\end{equation}
\end{remark}
\begin{theorem}\label{T3} Let $\lambda>0.$ Then the following generating function
\begin{equation}\label{100}
\sum_{k=0}^\infty {}_p\Psi_q\Big[_{(\lambda, A), ({\bf b}_{q-1}, {\bf B}_{q-1})}^{(\lambda+k, A), ({\bf a}_{p-1}, {\bf A}_{p-1})}\Big|z \Big]\frac{t^k}{k!}=(1-t)^{-\lambda}.{}_{p-1}\Psi_{q-1}\Big[_{ ({\bf b}_{q-1}, {\bf B}_{q-1})}^{({\bf a}_{p-1}, {\bf A}_{p-1})}\Big|\frac{z}{(1-t)^A} \Big],
\end{equation}
holds true for all $|t|<1$ and $z\in\mathbb{C}.$
\end{theorem}

\begin{proof}For convenience, let the left-hand side of the formula (\ref{100}) of Theorem \ref{T3}  be denoted by $\mathcal{S}.$ Then,  by substituting the series expression from (\ref{11}) into $\mathcal{S}$ and applying the binomial expansion (\ref{!!}) we obtain 
\begin{equation}
\begin{split}
\mathcal{S}&=\sum_{k=0}^\infty\left[\sum_{n=0}^\infty\frac{\Gamma(\lambda+k+nA)\prod_{l=1}^{p-1}\Gamma(a_l+nA_l)}{\Gamma(\lambda+nA)\prod_{m=1}^{q-1}\Gamma(b_m+nB_m)}\frac{z^k}{k!}\right]\frac{t^k}{k!}\\
&=\sum_{n=0}^\infty\frac{\prod_{l=1}^{p-1}\Gamma(a_l+nA_l)}{\prod_{m=1}^{q-1}\Gamma(b_m+nB_m)}\left[\sum_{k=0}^\infty\frac{\Gamma(\lambda+k+nA)}{\Gamma(\lambda+nA)}\frac{t^k}{k!}\right]\frac{z^n}{n!}\\
&=\sum_{n=0}^\infty\frac{\prod_{l=1}^{p-1}\Gamma(a_l+nA_l)}{\prod_{m=1}^{q-1}\Gamma(b_m+nB_m)}\left[\sum_{k=0}^\infty \binom{\lambda+k+n A-1}{k}t^k\right]\frac{z^n}{n!}\\
&=(1-t)^{-\lambda}\sum_{n=0}^\infty\frac{\prod_{l=1}^{p-1}\Gamma(a_l+nA_l)}{n!\prod_{m=1}^{q-1}\Gamma(b_m+nB_m)}\left(\frac{z}{(1-t)^A}\right)^n\\
&=(1-t)^{-\lambda}{}_{p-1}\Psi_{q-1}\Big[_{ ({\bf b}_{q-1}, {\bf B}_{q-1})}^{({\bf a}_{p-1}, {\bf A}_{p-1})}\Big|\frac{z}{(1-t)^A} \Big],
\end{split}
\end{equation}
which shows that the generating function (\ref{100}) holds true. This completes the proof of Theorem \ref{T3}.
\end{proof}
\begin{remark} If we setting $p=q=1$ in the above Theorem and the same steps as in the proof of Theorem \ref{T3} we get for all $\lambda>0$
\begin{equation}\label{2.38}
\sum_{k=0}^\infty {}_1\Psi_1\Big[_{(\lambda, A)}^{(\lambda+k, A)}\Big|z \Big]\frac{t^k}{k!}=(1-t)^{-\lambda}\exp\left(\frac{z}{(1-t)^A}\right),
\end{equation} 
where $|t|<1.$ In particular, for all $\lambda>0$ we have
\begin{equation}
\sum_{k=0}^\infty {}_1F_1\Big[_{\;\;\;\lambda}^{\;\lambda+k}\Big|z \Big]\frac{\Gamma(\lambda+k)t^k}{k!}=\frac{\Gamma(\lambda) }{(1-t)^{\lambda}}\exp\left(\frac{z}{1-t}\right),
\end{equation} 
where $|t|<1.$
\end{remark}
\begin{remark} Setting $A_i=Bj=1$ we obtain for all $a_i, b_j>0$ and $|t|<1$ the following generating function (known or new) for the hypergeometric function 
\begin{equation}\label{100789}
\sum_{k=0}^\infty {}_pF_q\Big[_{a_1, b_2,...,b_q}^{a_1+k,a_2,...a_p}\Big|z \Big]\frac{\Gamma(a_1+k)t^k}{k!}=\frac{\Gamma(a_1)}{(1-t)^{a_1}}.{}_{p-1}F_{q-1}\Big[_{b_2,...,b_q}^{a_2,...,a_p}\Big|\frac{z}{(1-t)} \Big].
\end{equation}
\end{remark}
\section{Applications: Generating functions for a certain classes of the generalized Mathieu-type series and the extended Hurwitz--Lerch zeta functions} The main object of this section is to investigate several generating functions for a
certain class of  generalized Mathieu-type series and The extended Hurwitz--Lerch zeta function. Here, and in what follows, the generalized Mathieu-type series is defined by \cite{ZAZ}:
\begin{equation}\label{math}
S_\mu^{(\alpha,\beta)}(r;\textbf{a})=S_\mu^{(\alpha,\beta)}(r;\{a_k\}_{k=0}^\infty)=\sum_{k=1}^\infty\frac{2a_k^\beta}{(r^2+a_k^\alpha)^\mu},\;(r,\alpha,\beta,\mu>0),
\end{equation}
where it is tacitly assumed that the positive sequence  $$\textbf{a}=(a_k)_k,\;\textrm{such\;that}\;\lim_{k\longrightarrow\infty}a_k=\infty,$$
is so chosen  that the infinite series in the definition (\ref{math}) converges, that is, that the following auxiliary series:
$$\sum_{k=0}^\infty \frac{1}{a_k^{\mu\alpha-\beta}}$$
is convergent. 
\begin{theorem}Let $\alpha>0$ and $\mu>1.$ Then for $r>0$ and $x>0$ there holds the formula
\begin{equation}\label{1002}
\sum_{k=0}^\infty \Gamma(\mu+k) S_{\mu+k}^{(\alpha,k\alpha)}(r;\{n^{\frac{1}{\alpha}}\}_{n=1}^\infty)\frac{t^k}{k!}=\frac{2\Gamma(\mu)}{(1-t)^\mu}\zeta\left(\mu, 1+\frac{r^2}{1-t}\right),
\end{equation}
where $\zeta(s,a)$ is the Hurwitz Zeta Function defined by:
$$\zeta(s,a)=\sum_{n=0}^\infty\frac{1}{(n+a)^s},\;\Re(s)>1.$$
\end{theorem}
\begin{proof}We make use  the representation integral for the Mathieu's series \cite{KZ}, 
$$S_{\mu}^{(\alpha,\beta)}(r;\{k^\nu\}_{k=1}^\infty)=\frac{2}{\Gamma(\mu)}\int_0^\infty \frac{x^{\nu(\mu \alpha-\beta)-1}}{e^{x}-1}{}_1\Psi_1\left(^{(\mu,1)}_{(\nu(\mu \alpha-\beta),\nu\alpha)}\Big|-r^2x^{\nu\alpha}\right)dx,$$
combining with (\ref{2.38}), we have
\begin{equation}\label{102}
\begin{split}
\sum_{k=0}^\infty \Gamma(\mu+k) S_{\mu+k}^{(\alpha,k\alpha)}(r;\{n^{\frac{1}{\alpha}}\}_{n=1}^\infty)\frac{t^k}{2k!}&=\int_0^\infty\frac{x^{\mu-1}}{e^{x}-1}\left[\sum_{k=0}^\infty{}_1\Psi_1\left[^{(\mu+k,1)}_{(\mu,1)}\Big|-r^2 x\right]\frac{t^k}{k!}\right] dx\\
&=(1-t)^{-\lambda}\int_0^\infty\frac{x^{\mu-1}}{e^{x}-1}{}_{0}\Psi_{0}\Big[_{-}^{-}\Big|\frac{-r^2 x}{(1-t)} \Big]dx\\
&=(1-t)^{-\lambda}\int_0^\infty\frac{x^{\mu-1}}{e^{x}-1} e^{\frac{-r^2 x}{(1-t)}} dx.
\end{split}
\end{equation}
We now make use of the following known formula \cite[Eq. (10), p. 144]{7}
$$\int_0^\infty \frac{x^{\nu-1}e^{-px}}{(1-e^{-\frac{x}{\alpha}})}dx=\alpha^\nu\Gamma(\nu)\zeta(\nu, \alpha p), (\nu>1, p>0),$$
Inserting the above result with the help of (\ref{102}), the results (\ref{1002}) readily follows.
\end{proof}
\begin{corollary}\label{C20}Let $\alpha>0$ and $\mu>1.$ Then the following formula
\begin{equation}\label{10020}
\sum_{k=0}^\infty S_{\mu+k}^{(\alpha,k\alpha)}(\sqrt{1-t};\{n^{\frac{1}{\alpha}}\}_{n=1}^\infty)\frac{\Gamma(\mu+k) t^k}{k!}=\frac{2\Gamma(\mu)}{(1-t)^\mu}\zeta\left(\mu,2\right),
\end{equation}
holds true for all $|t|<1.$ Moreover, the following double series identity holds true:
\begin{equation}\label{100200}
\sum_{k=0}^\infty\sum_{n=1}^\infty\frac{(k+1)}{(1+mn)^2}\left(\frac{(m-1)n}{(1+mn)}\right)^k=\frac{\pi^2-6}{6},\;m\geq2.
\end{equation}
\end{corollary}
\begin{proof} Letting $r=\sqrt{1-t}$ in (\ref{1002}) we easily get the formula (\ref{10020}). Next, setting $t=1-\frac{1}{m},\;m\geq2$ and $\mu=2$ in (\ref{10020}) we find
\begin{equation}\label{ya rab}
\sum_{k=0}^\infty \frac{(k+1)}{2m^2}S_{2+k}^{(\alpha,k\alpha)}(m^{-\frac{1}{2}};\{n^{\frac{1}{\alpha}}\}_{n=1}^\infty)\left(\frac{m-1}{m}\right)^k=\zeta\left(2,2\right)=\zeta(2)-1,
\end{equation}
where $\zeta(s)$ is the Riemann zeta function defined by $$\zeta(s)=\sum_{k=1}^\infty\frac{1}{k^s},\;s>1.$$
Now, combining (\ref{ya rab}) with the definition of the generalized Mathieu series (\ref{math}), and using the fact that $\zeta(2)=\frac{\pi^2}{6},$  we obtain the desired formula (\ref{100200}) and consequently the proof of Corollary \ref{C20} is complete.
\end{proof}
\begin{remark}Setting in the formula (\ref{100200}), $m=2, m=3$ and $m=4$ respectively, we get the following double series identities
\begin{equation}
\sum_{k=0}^\infty\sum_{n=1}^\infty\frac{(k+1)}{(1+2n)^2}\left(\frac{n}{1+2n}\right)^k=\frac{\pi^2-6}{6},
\end{equation}
\begin{equation}
\sum_{k=0}^\infty\sum_{n=1}^\infty\frac{(k+1)}{(1+3n)^2}\left(\frac{2n}{1+3n}\right)^k=\frac{\pi^2-6}{6},
\end{equation}
\begin{equation}
\sum_{k=0}^\infty\sum_{n=1}^\infty\frac{(k+1)}{(1+4n)^2}\left(\frac{3n}{1+4n}\right)^k=\frac{\pi^2-6}{6}.
\end{equation}
\end{remark}
\begin{remark}If we set $\mu=3$ (respectively $m=4$) and $t=1-\frac{1}{m}, \;m\geq2$ in the formula (\ref{10020}) we obtain the following formulas
\begin{equation}
\sum_{k=0}^\infty\sum_{n=1}^\infty\frac{(k+1)(k+2)}{(1+mn)^3}\left(\frac{(m-1)n}{(1+mn)}\right)^k=\zeta(3,2)=\zeta(3)-1\approx 0,202056903.
\end{equation}
and
\begin{equation}
\sum_{k=0}^\infty\sum_{n=1}^\infty\frac{(k+1)(k+2)(k+3)}{(1+mn)^4}\left(\frac{(m-1)n}{(1+mn)}\right)^k=\zeta(4,2)=\zeta(4)-1=\frac{\pi^4-90}{90}.
\end{equation}
\end{remark}

The extended Hurwitz-Lerch zeta function
\begin{equation}
\begin{split}
\Phi^{(\lambda_j,\rho_j;p)}_{(\mu_j,\sigma_j;q)}(z,s,a)&=\Phi^{(\rho_1,...,\rho_p;\sigma_1,...,\sigma_q)}_{\lambda_1,...,\lambda_p;\mu_1,...,\mu_q}(z,s,a)\\
&=\left(\frac{\prod_{j=1}^q\Gamma(\mu_j)}{\prod_{j=1}^p\Gamma(\lambda_j)}\right)\sum_{k=0}^\infty\frac{\prod_{j=1}^p\Gamma(\lambda_j+k\rho_j)}{\prod_{j=1}^q\Gamma(\mu_j+k\sigma_j)}\frac{z^k}{k!(k+a)^s}
\end{split}
\end{equation}
\begin{equation*}
\begin{split}
\Big(p, q&\in\mathbb{N}_0;\lambda_j\in\mathbb{C}\;(j=1,...,p); a,\mu_j\in\mathbb{C} \setminus\mathbb{Z}_0^-\; ( j = 1 , . . . , q);\\
\rho_j&,\sigma_k\in\mathbb{R}^+ ( j = 1 , . . . , p ; k = 1 , . . . , q );\\
\Delta_1&>-1\;\;\textrm{when}\;\; s , z\in\mathbb{C};\\
\Delta_1&=-1\;\;\textrm{and}\;\; s\in\mathbb{C}\;\textrm{when}\;|z|<\nabla^*;\\
\Delta_1&=-1\;\;\textrm{and}\;\; \Re(\Xi)>\frac{1}{2}\;\textrm{when}\;|z|<\nabla^*\Big).
\end{split}
\end{equation*}
where
$$\nabla^*=\left(\prod_{j=1}^p\rho_j^{-\rho_j}\right).\left(\prod_{j=1}^q\sigma_j^{\sigma_j}\right)$$
and
$$\Delta_1=\sum_{j=1}^q\sigma_j-\sum_{j=1}^p\rho_j,\;\;\textrm{and}\;\;\Xi=s+\sum_{j=1}^q\mu_j-\sum_{j=1}^p\lambda_j+\frac{p-q}{2}.$$
Moreover, the extended Hurwitz-Lerch zeta function possesses the following integral representation
\begin{equation}\label{LURCH}
\Phi^{(\lambda_j,\rho_j;p)}_{(\mu_j,\sigma_j;q)}(z,s,a)=\left(\frac{\prod_{j=1}^q\Gamma(\mu_j)}{\Gamma(s)\prod_{j=1}^p\Gamma(\lambda_j)}\right)\int_0^\infty \xi^{s-1} e^{-a\xi}{}_p\Psi_q\Big[_{(\mu_1,\sigma_1),...,(\mu_q,\sigma_q)}^{(\lambda_1,\rho_1),...,(\lambda_p,\rho_p)}\Big|z e^{-\xi} \Big]d\xi
\end{equation}
$$\left(\;\min(\Re(a), \Re(s))>0\;\right).$$
\begin{theorem}\label{T789} The following generating function
\begin{equation}\label{YARABBI}
\sum_{k=0}^\infty \Phi^{(\rho_1,...,\rho_p;\rho_1,\sigma_2,...,\sigma_q)}_{\lambda_1+k,\lambda_2,...,\lambda_p;\lambda_1,\mu_2,...,\mu_q}(z,s,a)\frac{\Gamma(\lambda_1+k) t^k}{k!}=\frac{\Gamma(\lambda_1)}{(1-t)^{\lambda_1}}\Phi^{(\rho_2,...,\rho_p;\sigma_2,...,\sigma_q)}_{\lambda_2,...,\lambda_p;\mu_2,...,\mu_q}(z(1-t)^{-\rho_1},s,a)
\end{equation}
holds true for all $|t|<1.$
\end{theorem}
\begin{proof}In virtue of (\ref{LURCH}) and (\ref{100}) we get
\begin{equation*}
\begin{split}
\sum_{k=0}^\infty \Phi^{(\rho_1,...,\rho_p;\rho_1,\sigma_2,...,\sigma_q)}_{\lambda_1+k,\lambda_2,...,\lambda_p;\lambda_1,\mu_2,...,\mu_q}(z,s,a)&\frac{\Gamma(\lambda_1+k) t^k}{k!}=\frac{\Gamma(\lambda_1)\prod_{j=2}^p\Gamma(\mu_j)}{\Gamma(s)\prod_{j=2}^p\Gamma(\lambda_j)}\\
&\times\int_0^\infty \xi^{s-1} e^{-a\xi}\left[\sum_{k=0}^\infty {}_p\Psi_q\Big[_{(\lambda_1,\rho_1),(\mu_2,\sigma_2),...,(\mu_q,\sigma_q)}^{(\lambda_1+k,\rho_1),(\lambda_2,\rho_2),...,(\lambda_p,\rho_p)}\Big|z e^{-\xi} \Big]\frac{t^k}{k!}\right]d\xi\\
&=\frac{\Gamma(\lambda_1)\prod_{j=2}^p\Gamma(\mu_j)}{\Gamma(s)(1-t)^{\lambda_1}\prod_{j=2}^p\Gamma(\lambda_j)}\\
&\times\int_0^\infty \xi^{s-1} e^{-a\xi}{}_{p-1}\Psi_{q-1}\Big[_{(\mu_2,\sigma_2),...,(\mu_q,\sigma_q)}^{(\lambda_2,\rho_2),...,(\lambda_p,\rho_p)}\Big|\frac{z}{(1-t)^{\rho_1}e^{\xi}} \Big]d\xi.
\end{split}
\end{equation*}
Combining this with (\ref{LURCH}) yields to the desired assertion (\ref{YARABBI}) of Theorem \ref{T789}.
\end{proof}
 
The Hurwitz-Lerch zeta function $\Phi(z,s,a)$ is defined by
\begin{equation}
\Phi(z,s,a)=\sum_{n=0}^\infty\frac{z^n}{(n+a)^s}
\end{equation}
$$\left(\;H_3: a\in\mathbb{C}\setminus\mathbb{Z}_0^-; s\in\mathbb{C}\;\textrm{when}\;|z|<1; \Re(s)>1\;\textrm{when}\;|z|=1\right).$$

The Hurwitz-Lerch zeta function itself reduces not only to  the Riemann zeta function $\zeta(s),$  the Hurwitz zeta function  $\zeta(s,a),$  but also to such other important  functions of Analytic Number Theory as as the Polylogarithm function (or de Jonqui\`ere's function) $\Li_s(z)$, the Lipschitz-Lerch zeta function $L(\xi, a,s)$ and the Lerch zeta function $l_s(\xi)$ defined by \cite[Chapter 1, p. 27-31]{E1}
\begin{equation}
\Li_s(z)=\sum_{n=1}^\infty\frac{z^n}{n^s},\;\;\left(\Re(s)>0; z\in\mathbb{C}\;\;\textrm{when}\;|z|<1\right)
\end{equation}
\begin{equation}
L(\xi, a,s)=\sum_{n=0}^\infty\frac{e^{2i n\pi\xi}}{(n+a)^s},\;\;\left(\Re(s)>1; \xi\in\mathbb{R}; 0<a\leq1\right).
\end{equation}
and
\begin{equation}
l_s(\xi)=\sum_{n=0}^\infty\frac{e^{2i n\pi\xi}}{(n+1)^s},\;\;\left(\Re(s)>1; \xi\in\mathbb{R}\right).
\end{equation}
\begin{corollary}\label{CV78}
The following generating function
\begin{equation}\label{YARABBIlotfek}
\sum_{k=0}^\infty \Phi^{(\rho_1,1;\rho_1)}_{\lambda_1+k,1;\lambda_1}(z,s,a)\frac{\Gamma(\lambda_1+k) t^k}{k!}=\frac{\Gamma(\lambda_1)}{(1-t)^{\lambda_1}}\Phi(z(1-t)^{-\rho_1},s,a)
\end{equation}
holds true for all $-1<t<0.$ Furthermore, the following generating function involving the Lipschitz-Lerch zeta function $L(\xi, a,s)$
\begin{equation}\label{7biba}
\sum_{k=0}^\infty \Phi^{(\rho_1,1;\rho_1)}_{\lambda_1+k,1;\lambda_1}(e^{2i\pi\xi},s,a)\frac{\Gamma(\lambda_1+k) t^k}{k!}=\frac{\Gamma(\lambda_1)}{(1-t)^{\lambda_1}}L(\xi, a,s),
\end{equation}
holds true for all $-1<t<0, 0<a\leq1, \Re(s)>1$ and $\xi\in\mathbb{R}.$
\end{corollary}
\begin{proof}Letting $p=2, q=1$ and $\lambda_2=1$ in (\ref{YARABBI}) we obtain
\begin{equation}
\begin{split}
\sum_{k=0}^\infty \Phi^{(\rho_1,1;\rho_1)}_{\lambda_1+k,1;\lambda_1}(z,s,a)\frac{\Gamma(\lambda_1+k) t^k}{k!}&=\frac{\Gamma(\lambda_1)}{(1-t)^{\lambda_1}}\Phi^{(1;-)}_{1;-}(z(1-t)^{-\rho_1},s,a)\\
&=\frac{\Gamma(\lambda_1)}{(1-t)^{\lambda_1}}\sum_{n=0}^\infty\frac{(z(1-t)^{-\rho_1})^n}{(n+a)^s}\\
&=\frac{\Gamma(\lambda_1)}{(1-t)^{\lambda_1}}\Phi\left(z(1-t)^{-\rho_1},s,a\right),
\end{split}
\end{equation}
and consequently the formula (\ref{YARABBIlotfek}) holds true. Finally, setting $z=e^{2i\pi\xi}$ in (\ref{YARABBIlotfek}) we find (\ref{7biba}), which completes the proof of Corollary \ref{CV78}.
\end{proof}
\begin{corollary} Assume that the Hypotheses $(H_3)$ are satisfied. Then the following  double series identities
\begin{equation}\label{4560}
\sum_{k=0}^\infty\sum_{n=0}^\infty\frac{\Gamma(n+k+1)}{n!}\frac{(1-t)^{n+1}z^n t^k}{(n+a)^s}=\Phi\left(z,s,a\right),
\end{equation}
holds true for all $-1<t<0$ In particular, the following  double series identities
\begin{equation}\label{4561}
\sum_{k=0}^\infty\sum_{n=0}^\infty\frac{\Gamma(n+k+1)}{n!}\frac{(1-t)^{n+1} e^{2in\pi\xi}t^k}{(n+a)^s}=L(\xi, a,s),
\end{equation}
holds true for all $-1<t<0, \xi\in\mathbb{R}, \Re(s)>1$ and $0<a\leq1.$
\end{corollary}
\begin{proof}Upon setting $\lambda_1=\rho_1=1$ and replace $z$ by $(1-t)z$ in (\ref{YARABBIlotfek}) and straightforward calculation would yield to the formula (\ref{4560}). Now, letting $z=e^{2i\pi\xi}$ in (\ref{4560}) we obtain (\ref{4561}).
\end{proof}
\begin{remark} Using the fact that $\Li_s(z)=z\Phi\left(z,s,1\right)$ we get
\begin{equation}\label{4560}
\sum_{k=0}^\infty\sum_{n=0}^\infty\frac{\Gamma(n+k+1)}{n!}\frac{(1-t)^{n+1}z^n t^k}{(n+1)^s}=\frac{\Li_s(z)}{z},
\end{equation}
holds true for all $-1<t<0, \Re(s)>0$ and $z\in\mathbb{C}$ when $|z|<1.$ In particular, by using  some particular expressions of the Polylogarithm function we obtain for $-1<t<0:$
\begin{eqnarray}
\sum_{k=0}^\infty\sum_{n=0}^\infty\frac{\Gamma(n+k+1)(1-t)^{n+1}z^n t^k}{(n+1)!}&=&\frac{-\log(1-z)}{z}, 0<z<1\\
\sum_{k=0}^\infty\sum_{n=0}^\infty\frac{\Gamma(n+k+1)(1-t)^{n+1}t^k}{(n+1)! 2^n}&=&2\log(2)\\
\sum_{k=0}^\infty\sum_{n=0}^\infty\frac{\Gamma(n+k+1)(1-t)^{n+1}t^k}{(n+1)! (n+1) 2^n}&=&\frac{\pi^2}{6}-\log^2(2)\\
\sum_{k=0}^\infty\sum_{n=0}^\infty\frac{\Gamma(n+k+1)(1-t)^{n+1}t^k}{(n+1)! (n+1)^2 2^n}&=&\frac{\log^3(2)}{3}-\frac{\pi^2}{6}\log(2)-\frac{\pi^2}{6}\log(2)+\frac{7}{8}\zeta(3).
\end{eqnarray}
\end{remark}

\end{document}